\documentclass[11pt]{article}
\usepackage{latexsym}

\usepackage{amsmath,amsthm}
\usepackage{amsfonts}
\usepackage{amssymb}
\usepackage{a4}
\usepackage [all] {xy}

\newtheorem{thm}{Theorem}[section]
\newtheorem{cor}[thm]{Corollary}
\newtheorem{lem}[thm]{Lemma}
\newtheorem{prop}[thm]{Proposition}

\theoremstyle{definition}

\newtheorem{exmp}{Example}[section]

\theoremstyle{remark}

\numberwithin{equation}{section}

\renewcommand{\div}{\operatorname{div}}

\newcommand{\R}{\mathbb{R}}

\newcommand{\supp}{\operatorname{supp}}

\newcommand{\pM}{{\partial M}}

\def\tilde{\widetilde}
\def \bfo {\begin {eqnarray*} }
\def \efo {\end {eqnarray*} }
\def \ba {\begin {eqnarray*} }
\def \ea {\end {eqnarray*} }
\def \beq {\begin {eqnarray}}
\def \eeq {\end {eqnarray}}
\def \supp {\hbox{supp }}

\def \det {\hbox{det}}

\def \p {\partial}

\def\tilde{\widetilde}
\def \bfo {\begin {eqnarray*} }
\def \efo {\end {eqnarray*} }
\def \ba {\begin {eqnarray*} }
\def \ea {\end {eqnarray*} }
\def \beq {\begin {eqnarray}}
\def \eeq {\end {eqnarray}}
\def \supp {\hbox{supp }}

\def \det {\hbox{det}}

\def \p {\partial}

\begin{document}

 \title{Reconstruction of Betti numbers of manifolds for anisotropic Maxwell and Dirac systems}

\author{Katsiaryna Krupchyk\footnote{Department of Mathematics and Statistics, 
        University of Helsinki,
         P.O. Box 68,
         FI-00014   Helsinki,
         Finland, \sf{katya.krupchyk@helsinki.fi}}\\
 Yaroslav Kurylev\footnote{Department of Mathematics,
University College London,
 Gower Street, London, WC1E 6BT, UK, \sf{Y.Kurylev@ucl.ac.uk}}
\\
 Matti Lassas,\footnote{
Department of Mathematics and Statistics, 
        University of Helsinki,
         P.O. Box 68,
         FI-00014   Helsinki,
         Finland, \sf{matti.lassas@helsinki.fi}}
}

\date{}

\maketitle



\maketitle



\begin{abstract}
We consider an invariant formulation of the system of Maxwell's
equations for an anisotropic medium on a compact orientable Riemannian $3$--manifold $(M,g)$ with nonempty
boundary. The system can be completed to  a Dirac type first order
system on the manifold.  We show that the
Betti numbers of the manifold can be recovered from the dynamical response operator for the Dirac system given on a part of the boundary. 
In the case of the original physical Maxwell system, assuming that the entire boundary is known, all Betti numbers of the manifold can also be determined from the dynamical response operator given on a part of the boundary. Physically, this operator maps  the tangential component of the electric field  into the tangential component of the magnetic field on the boundary.

\end{abstract}






\section{Introduction}

Recently there has been a lot of interest in inverse problems for Maxwell's equations in Euclidean domains in $\R^3$ and on compact Riemannian manifolds, see \cite{CarOlaSalo, KenSalUhl, KurLass06, KurLasSom06, OlaPaiSom1993, OlaPaiSom2003, OlaSom1996}. In a smooth bounded domain $M\subset \R^3$, Maxwell's equations are given by
\begin{equation}
\label{eq_Max_normal}
\begin{aligned}\textrm{curl}\, E(x,t)=&-B_t(x,t),\\
\textrm{curl}\, H(x,t)=&D_t(x,t),
\end{aligned}
\end{equation}
where $E$ and $H$ are the electric and magnetic fields, and $B$
and $D$ are the magnetic flux density and the electric
displacement. The fields $E$ and $D$,
and similarly, the fields $H$ and $B$ are related by the
constitutive relations,
\begin{equation}
\label{eq_constitutive}
 D(x,t)=\epsilon(x)E(x,t),\quad
B(x,t)=\mu(x) H(x,t),
\end{equation}
where the electric permittivity $\epsilon(x)$ and the magnetic
permeability $\mu(x)$ are $C^\infty$-smooth positive-definite
$3\times 3$-matrix valued functions on $M$.
The initial boundary value problem for the time dependent Maxwell's equations consists of \eqref{eq_Max_normal}, \eqref{eq_constitutive} together with the conditions
\begin{equation}
\label{eq_bound-cond}
\begin{aligned}
E(x,t)|_{t=-\tau_f}=0,\quad H(x,t)|_{t=-\tau_f}=0,\\
n\times E|_{\p M\times \R_-}=f,
\end{aligned}
\end{equation}
where $n$ is the unit exterior normal to $\p M$, and $\tau_f>0$ is such  that $f(x, t)=0$ for $t< -\tau_f$. 
The inverse problem associated with  \eqref{eq_Max_normal}, \eqref{eq_constitutive}, and \eqref{eq_bound-cond}, is the problem of reconstruction of 
electromagnetic parameters $\epsilon(x)$ and $\mu(x)$ from the knowledge of the response operator
\begin{equation}
\label{eq_response_int}
R:n\times E|_{\p M\times \R_-}\mapsto n\times H|_{\p M \times \R_-}. 
\end{equation}
From the point of view of modern electrodynamics and classical field theories,  it is natural to adopt an invariant approach to Maxwell's equations, where the domain $M$ is replaced by  a general $3$-dimensional 
smooth compact oriented connected Riemannian manifold, and the vector fields $E$, $H$, $D$, and $B$ are viewed as  differential forms, see \cite{Thi79}.  
The geometric inverse problem  is then to determine the unknown manifold $M$, together with the electromagnetic parameters, from the response operator  
\eqref{eq_response_int}, which is now defined in terms of boundary traces of the corresponding differential forms. 
See also  \cite{LasUhl01, LasTayUhl03}, where the problem of the reconstruction of a Riemannian manifold from  the Dirichlet-to-Neumann operator for harmonic functions, has been studied.

In the context of time-harmonic Maxwell's equations in an isotropic setting, i.e. when the parameters $\epsilon(x)$ and $\mu(x)$ are scalar,  the inverse problem for bounded domains in $\R^3$ was solved in \cite{OlaPaiSom1993}, see also \cite{ColPai92,  McD97, OlaSom1996}. Much less is known in the anisotropic case. To the best of our knowledge, the positive results in this direction have only been established in the case of an anisotropic medium of a special type, characterized by the polarization independent velocity of the wave propagation.
In terms of the electromagnetic parameters, this amounts to the existence of $\alpha(x)>0$ such that $\epsilon(x)=\alpha(x)\mu(x)$. 
In this case, under a certain geometric condition, 
it is shown in  \cite{KenSalUhl} that, if the
conformal class of $\epsilon(x)$ and $\mu(x)$ is known, the stationary boundary
measurements identify uniquely the conformal factors. 
There are also counterexamples for uniqueness of time-harmonic inverse problems involving very anisotropic and degenerate material parameters 
 \cite{GreKurLasUhl2009, GreKurLasUhl2007}. 
In \cite{KurLasSom06}, the inverse problem for Maxwell's equations in the time domain for an anisotropic medium was studied, still assuming that the wave propagation is independent of the polarization. It was shown that the Riemannian manifold and the electromagnetic parameters can be recovered from the dynamical response operator similar to \eqref{eq_response_int}, given on a finite time interval.  See also \cite{BelIsaPesShar200} for reconstruction of the wave speed.

In this paper we shall be concerned with the case of a general anisotropic medium. Specifically, working in the geometric setting of Maxwell's equations on a manifold $M$, we are able to recover the Betti numbers of the manifold from the dynamical response operator, given on an open subset of the boundary. This can be viewed as the first step in attempting to reconstruct  the geometry and topology of the underlying manifold, in the full generality of the anisotropic case.  Let us remark that in the isotropic case, as well as in the case when $\epsilon(x)=\alpha(x)\mu(x)$, $\alpha(x)>0$, the reconstruction of the manifold and the electromagnetic parameters is based on controllability results, which in turn rely crucially on generalizations of the Tataru unique continuation theorem \cite{EllNakTat02, KurLasSom06}. In our opinion, the main obstacle in the study of the inverse problem for the general anisotropic Maxwell system 
is due to the fact that such unique continuation results do not seem to be available in this case. 

We would like also to mention the paper \cite{BelSha2008}, where the reconstruction of the Betti numbers of a manifold from the Dirichlet-to-Neumann operator for the Hodge Laplacian on  differential forms is studied. 

The plan of the paper is as follows. Section 2 is devoted to the description of our geometric setup, including the completion of the Maxwell system to a Dirac type elliptic system, and contains the statement of the main results. 
We also discuss  examples that illustrate the significance of our results for the determination of the topological structure of an unknown object from the boundary measurements.
In Section 3, we prove the identifiability of the Betti numbers in the complete Maxwell  case, while in Section 4 we establish our results for the physical Maxwell system.

\section{Preliminaries and statement of the main results}

\subsection{Invariant definition of Maxwell's equations} 
Let $(M,g_0)$ be a
smooth compact oriented connected Riemannian 3-manifold $M$ with
$\partial M\not=\emptyset$. We shall first rewrite equations \eqref{eq_Max_normal}, \eqref{eq_constitutive}, in the anisotropic case,  using the language of differential forms . In doing so,  we shall follow closely \cite{KurLasSom06}, where the case $\epsilon(x)=\alpha(x)\mu(x)$, $\alpha(x)>0$, is considered.

Let  $\Lambda^k T^*M$, $k=0,1,\dots,3$,  be the bundle of the $k$-th exterior differential forms and $\Lambda T^* M$ be the full bundle of differential forms. 
Denote by $C^\infty(M,\Lambda^k T^*M)$ the space of smooth real exterior differential forms of degree $k$.

Define the fiberwise duality
between $1$-forms and vector fields,
\[
^ \flat:\ C^\infty (M, T M)\to C^\infty(M, \Lambda^1 T^*M),\quad X^\flat(Y)=g_0(X,Y),
\]
or in a coordinate system for $X=a^i\frac{\partial}{\partial
x^i}$, $X^\flat=g_{0,ij}a^jdx^i$. This map is bijective and has
the following properties \cite{Sch95}:
\[
(\textrm{curl}\,X)^\flat=*_0d X^\flat,\quad (\textrm{div}
X)^\flat=*_0d*_0X^\flat,
\]
where 
\[
d: C^\infty(M, \Lambda^k T^* M)\to C^\infty(M, \Lambda^{k+1} T^ *M)
\] 
is the exterior differential
and $*_0$ is the Hodge operator with respect to the metric $g_0$,
acting fiberwise, 
\[
*_0:C^\infty(M, \Lambda^k T^ *M) \to C^\infty(M, \Lambda^{3-k} T^ *M).
\]

We define the $1$-forms $\mathcal{E}=E^\flat$ and
$\mathcal{H}=H^\flat$ and the $2$-forms $\mathcal{B}=*_0B^\flat$
and $\mathcal{D}=*_0D^\flat$. Using the identity
$*_0*_0=\text{id}$, valid in the $3$-dimensional case, we can write
Maxwell's equations \eqref{eq_Max_normal} in terms of differential forms as
\begin{equation}
\label{eq_maxwell_diff_forms}
d\mathcal{E}=-\partial_t\mathcal{B},\quad
d\mathcal{H}=\partial_t\mathcal{D}.
\end{equation}

Consider now the constitutive relations \eqref{eq_constitutive}.
We 
shall determine a metric $g_\epsilon$ such that the
Hodge operator with respect to this metric, denoted by
$*_\epsilon$, satisfies 
\begin{equation}
\label{eq_cons_forms_1} \mathcal{D}=*_0(\epsilon
E)^\flat=*_\epsilon\mathcal{E}.
\end{equation}
In local coordinates $(x^1,x^2,x^3)$, we have  $\epsilon
E=\epsilon^i_k E^k\frac{\partial}{\partial x^i}$, $(\epsilon
E)^\flat=g_{0,ij}\epsilon^j_k E^kdx^i$ and thus, the middle term
of \eqref{eq_cons_forms_1} yields
\begin{align*}
*_0(\epsilon
E)^\flat&=*_0(g_{0,ij}\epsilon^j_kE^kdx^i)=\frac{1}{2}
\sqrt{\det(g_0)}g_0^{il}g_{0,ij}\epsilon_k^jE^ks_{lpq}dx^p\wedge
dx^q\\
&=\frac{1}{2} \sqrt{\det(g_0)}\epsilon_k^jE^ks_{jpq}dx^p\wedge
dx^q,
\end{align*}
where $s_{lpq}$ is the Levi-Civita permutation symbol. The
right hand side of \eqref{eq_cons_forms_1} implies that 
\[
*_\epsilon\mathcal{E}=*_\epsilon(g_{0,ik}E^kdx^i)=\frac{1}{2}
\sqrt{\det(g_\epsilon)}g_\epsilon^{ij}g_{0,ik}E^ks_{jpq}dx^p\wedge
dx^q.
\]
Hence, \eqref{eq_cons_forms_1} is valid if we set
\[
\sqrt{\det(g_\epsilon)}g_\epsilon^{ij}g_{0,ik}=\sqrt{\det(g_0)}\epsilon^j_k.
\]
By taking the determinants of  both sides, we get 
\[
\sqrt{\det(g_\epsilon)}=\det(\epsilon)\sqrt{\det(g_0)}.
\]
 Defining
\[
g_\epsilon^{ij}=\frac{1}{\det(\epsilon)} \epsilon^j_kg_0^{ki},
\]
we see that 
\eqref{eq_cons_forms_1} is valid. Similarly, we see that
for the metric tensor $g_\mu^{ij}=\frac{1}{\det(\mu)}
\mu^j_kg_0^{ki}$,  we have
\[
\mathcal{B}=*_0(\mu H)^\flat=*_\mu\mathcal{H}.
\]
Hence, the constitutive relations take the form
\[
\mathcal{D}(x,t)=*_\epsilon\mathcal{E}(x,t),\quad
\mathcal{B}(x,t)=*_\mu\mathcal{H}(x,t).
\]
We consider the waves that satisfy the initial conitions
\[
B(x,t)|_{t=-\tau}=0,\quad D(x,t)|_{t=-\tau}=0, \quad \tau>0,
\]
Applying the divergence operator  to \eqref{eq_Max_normal}, we have
\[
\div B(x,t)=0,\ \div D(x,t)=0,\quad t\in \R,\quad  x\in M.
\]
In terms of differential forms these equations imply that 
\begin{equation}
\label{eq_compat} d\mathcal{B}=0,\quad d\mathcal{D}=0.
\end{equation}

In the further considerations, we will use only the pair
$(\mathcal{E},\mathcal{B})$ and denote it by
$(\omega^1,\omega^2)$, where $\omega^1=\mathcal{E}$
and $\omega^2=\mathcal{B}$. 
The compatibility conditions \eqref{eq_compat} imply that
\begin{equation}
\label{eq_max1} d\omega^2=0,\quad d*_\epsilon \omega^1=0.
\end{equation}
It follows from \eqref{eq_maxwell_diff_forms} that 
\begin{equation}
\label{eq_max2} \omega^2_t=-d\omega^1,\quad \omega^1_t=*_\epsilon
d*_\mu\omega^2.
\end{equation}
Let us consider the following codifferentials,
\begin{equation}
\label{eq_codiff}
\delta_{\epsilon,\mu}\omega^2=*_\epsilon d*_\mu
\omega^2, \quad \delta_{\mu,\epsilon}\omega^k=-*_\mu
d*_\epsilon \omega^k, k=1,3.
\end{equation}
Then \eqref{eq_max1} and \eqref{eq_max2} yield
\begin{equation}
\label{eq_max}
\begin{aligned} &\omega^1_t=\delta_{\epsilon,\mu}\omega^2,\quad \delta_{\mu,\epsilon}\omega^1=0,\\
& \omega^2_t=-d\omega^1,\quad d\omega^2=0.
\end{aligned}
\end{equation}
These equations are called \emph{Maxwell's equations for forms in
the divergence free case} on a Riemannian manifold $M$.

We shall now extend the above equations to the full bundle of exterior
differential forms $\Lambda T^*M $. To this end, we
introduce auxiliary forms, $\omega^0\in C^\infty(M)$ and
$\omega^3\in C^\infty(M, \Lambda^3 T^*M)$, which vanish in the electromagnetic
theory, by
\[
\omega^0_t=\delta_{\mu,\epsilon}\omega^1,\quad
\omega^3_t=-d\omega^2.
\]
Since $\omega^0=0$ and $\omega^3=0$ in the electromagnetic theory, we can
modify equations \eqref{eq_max} to have
\begin{align*}
&\omega^1_t=-d\omega^0+\delta_{\epsilon,\mu}\omega^2,\quad \omega^3_t=-d\omega^2,\\
& \omega^2_t=-d\omega^1+\delta_{\mu,\epsilon}\omega^3,\quad
\omega^0_t=\delta_{\mu,\epsilon}\omega^1,
\end{align*}
or, in the matrix form,
\begin{equation}
\label{eq_complete_maxwell} \omega_t+D\omega=0,
\end{equation}
where
$\omega=(\omega^0,\omega^1,\omega^2,\omega^3)$
and the operator $D$ is given by
\begin{equation}
\label{eq_matrix}
D=\begin{pmatrix} 0&-\delta_{\mu,\epsilon} & 0 & 0\\
d & 0 & -\delta_{\epsilon,\mu} & 0\\
0& d & 0 & -\delta_{\mu,\epsilon}\\
0& 0 & d &0
\end{pmatrix}.
\end{equation}
Equations \eqref{eq_complete_maxwell}, \eqref{eq_matrix} are
called \emph{the complete Maxwell system}. Notice that the operator $D$ is of the Dirac type.

\subsection{Function spaces}

Define the $L^2$-inner product in the space $C^\infty(M, \Lambda^kT^* M)$ as follows,
\begin{align*}
(\omega^k,\eta^k)_{L^2_\mu}=&\int_M
\omega^k\wedge*_{\mu}\eta^k,\quad
k=0,2,\\
(\omega^k,\eta^k)_{L^2_\epsilon}=&\int_M
\omega^k\wedge*_{\epsilon}\eta^k,\quad
k=1,3,
\end{align*}
and denote by $L^2(M, \Lambda^k T^* M)$ the completion of $C^\infty(M, \Lambda^kT^* M)$ in
the corresponding norm. 
In the complexified case, we take the corresponding sesquilinear extension of the inner product.  We denote by $H^s(M, \Lambda^k T^*  M)$ the standard Sobolev space of $k$-forms. 

The natural domain of the exterior differential $d$ in $L^2(M, \Lambda^k T^* M)$ is 
\[
H(d,\Lambda^k T^* M)=\{\omega^k\in L^2(M, \Lambda^k T^* M):d\omega^k\in L^2(M, \Lambda^{k+1} T^* M)\},
\]
and we define 
\[
H(\delta_{\epsilon,\mu},\Lambda^k T^* M)=\{\omega^k\in L^2(M, \Lambda^k T^* M):\delta_{\epsilon,\mu}\omega^k\in L^2(M, \Lambda^{k-1} T^* M)\},
\]
and similarly for $\delta_{\mu,\epsilon}$. 

Let $i^*:C^\infty(M, \Lambda^k T^* M)\to C^\infty(\p M, \Lambda^k T^* M)$ be the pull-back of the
 imbedding $i:\partial M\to M$. Then we define the
\emph{tangential trace} of $k$-forms as
\[
\mathbf{t}:C^\infty(M, \Lambda^k T^* M)\to C^\infty(\p M, \Lambda^k T^* M), \quad
\mathbf{t}\omega^k=i^*\omega^k,\quad 
k=0,1,2,
\]
and the \emph{normal trace} as
\begin{align*}
\mathbf{n}:\ C^\infty(M, \Lambda^k T^* M)& \to C^\infty(\p M, \Lambda^{3-k} T^* M),\\
\mathbf{n}\omega^k=i^*(*_\epsilon \omega^k),&\
 k=1,3,\quad 
\mathbf{n}\omega^2=i^*(*_\mu \omega^2).
\end{align*}
Set
\[
\langle\mathbf{t}\omega^k,\mathbf{n}\eta^{k+1}\rangle=\int_{\partial
M}\mathbf{t}\omega^k\wedge\mathbf{n}\eta^{k+1},\quad k=0,1,2.
\]
With this notation, Stokes' formulae for differential forms can be written as
\begin{equation}
\label{eq_stokes_formulae}
\begin{aligned}
(d\omega^0,\eta^1)_{L^2_\epsilon}-(\omega^0,\delta_{\mu,\epsilon}\eta^1)_{L^2_\mu}=
\langle\mathbf{t}\omega^0,\mathbf{n}\eta^1\rangle,\\
(d\omega^1,\eta^2)_{L^2_\mu}-(\omega^1,\delta_{\epsilon,\mu}\eta^2)_{L^2_\epsilon}=
\langle\mathbf{t}\omega^1,\mathbf{n}\eta^2\rangle,\\
(d\omega^2,\eta^3)_{L^2_\epsilon}-(\omega^2,\delta_{\mu,\epsilon}\eta^3)_{L^2_\mu}=
\langle\mathbf{t}\omega^2,\mathbf{n}\eta^3\rangle.
\end{aligned}
\end{equation}
Using \eqref{eq_matrix} and \ref{eq_stokes_formulae}, we get 
\begin{equation}
\label{eq_stokes_D}
(D\omega,\eta)_{L^2}+(\omega,D\eta)_{L^2}=\langle\mathbf{t}\omega,\mathbf{n}\eta\rangle+
\langle\mathbf{t}\eta,\mathbf{n}\omega\rangle,
\end{equation}
where $\mathbf{t}\omega=(\mathbf{t}\omega^0,\mathbf{t}\omega^1,\mathbf{t}\omega^2)$,
$\mathbf{n}\omega=(\mathbf{n}\omega^1,\mathbf{n}\omega^2,\mathbf{n}\omega^3)$,
and
\[
\langle\mathbf{t}\omega,\mathbf{n}\eta\rangle=\langle\mathbf{t}\omega^0,\mathbf{n}\eta^1\rangle+
\langle\mathbf{t}\omega^1,\mathbf{n}\eta^2\rangle+\langle\mathbf{t}\omega^2,\mathbf{n}\eta^3\rangle.
\]
Here we take $\omega,\eta\in \mathcal{H}$, where 
\begin{align*}
\mathcal{H}&=H(d, \Lambda^0 T^* M)\times [H(d,  \Lambda^1 T^* M)\cap H(\delta_{\mu,\epsilon}, \Lambda^1 T^* M)]\\
&\times
 [H(d,  \Lambda^2 T^* M)\cap H(\delta_{\epsilon,\mu}, \Lambda^2 T^* M)]
 \times H(\delta_{\mu,\epsilon}, \Lambda^3 T^* M).
\end{align*}

It will be convenient to write $\delta$ to stand for both $\delta_{\mu,\epsilon}$ and $\delta_{\epsilon,\mu}$, when no risk of ambiguity is possible. 
There are well defined extensions of the boundary trace operators $\mathbf{t}$ and $\mathbf{n}$ to the spaces  $H(d,\Lambda^k T^*M)$ and $H(\delta,\Lambda^k T^*M)$, see  \cite{Paquet}.

\begin{lem} The operators $\mathbf{t}$ and $\mathbf{n}$ can be extended to continuous surjective maps
\begin{equation}
\label{eq_t_1}
\mathbf{t}:H(d,\Lambda^k T^*M)\to H^{-1/2}(d,\p M, \Lambda^k T^* M),
\end{equation}
\begin{equation}
\label{eq_n_1}
\mathbf{n}:H(\delta,\Lambda^{k+1} T^*M)\to H^{-1/2}(d,\p M, \Lambda^{2-k} T^* M),
\end{equation}
where $H^{-1/2}(d,\p M, \Lambda^kT^* M)$ is given by
\[
\{\omega^k\in H^{-1/2}(\p M,\Lambda^k T^*  M):d\omega^k\in H^{-1/2}(\p M, \Lambda^{k+1} T^*  M)\}.
\]
\end{lem}

Let 
$
H_t(d,\Lambda^k T^* M)
$
stand for the kernel of  \eqref{eq_t_1}, and $H_n(\delta,\Lambda^{k+1} T^* M)$ will denote the kernel of the operator \eqref{eq_n_1}. 

Using \eqref{eq_stokes_formulae}, we can verify the following result in a standard way, see also \cite[Lemma 1.3]{KurLasSom06}. 

\begin{lem} 
\label{lem_adjoint}
The Hilbert space adjoint of 
\[
d:L^2(M,\Lambda^0 T^* M)\to L^2(M, \Lambda^1 T^*M),
\] 
equipped with the domain $H_t(d,\Lambda^0 T^* M)$, is the operator $\delta_{\mu,\epsilon}$ with the domain $H(\delta_{\mu,\epsilon}, \Lambda^1 T^* M)$. The  Hilbert space adjoint of 
\[
\delta_{\mu,\epsilon}:L^2(M, \Lambda^1 T^*M)\to L^2(M,\Lambda^0 T^* M),
\]
 equipped with the domain $H(\delta_{\mu,\epsilon},\Lambda^1 T^*M)$, is the operator $d$ with the domain $H_t(d,\Lambda^0 T^* M)$.

\end{lem}

It is clear that analogous statements hold for the operators $d$ and $\delta$, acting on forms of higher degree.

We shall need the following result. 

\begin{prop}

\begin{itemize}
\item [(i)] The operator $D$, given by \eqref{eq_matrix} and equipped with the domain 
\begin{align*}
\mathcal{D}(D)=&H_t(d,\Lambda^0 T^* M)\times [H_t(d, \Lambda^1 T^* M)\cap H(\delta_{\mu,\epsilon},\Lambda^1 T^* M)]\\
&\times
 [H_t(d, \Lambda^2 T^* M)\cap H(\delta_{\epsilon,\mu},\Lambda^2 T^* M)]
 \times H(\delta_{\mu,\epsilon},\Lambda^3 T^* M),
\end{align*}
is skew-adjoint on $L^2$. 

\item[(ii)] The spectrum of the operator  $D$ with the domain $\mathcal{D}(D)$ is  discrete.

\item[(iii)] The operator $D$ is an elliptic differential operator in the interior of $M$. 

\end{itemize}

\end{prop}

\begin{proof}

(i). Using the definition of the domain of the adjoint and  Lemma \ref{lem_adjoint}, we obtain that  $\mathcal{D}(D^ *)=\mathcal{D}(D)$. 
The skew-adjointness of $D$ then  follows from \eqref{eq_stokes_D}, which holds for $\omega,\eta\in \mathcal{D}(D)$. 

(ii). In view of Gaffney's inequality 
\cite[Corollary 2.1.6]{Sch95}, 
\[
H_t(d, \Lambda^k T^* M)\cap H(\delta,\Lambda^k T^* M)=\{
\omega^k\in H^1(M,\Lambda^k T^* M):\mathbf{t}\omega^k=0
\}, k=1,2,
\]
together with the Sobolev embedding, we conclude that the imbedding $\mathcal{D}(D) \hookrightarrow L^2$ is compact.  Hence,  the spectrum of  $\mathcal{D}$ is  discrete. 

(iii).  It suffices to show the ellipticity of $D^2$. Since $\delta_{\mu,\epsilon}\delta_{\epsilon,\mu}=0$ and
$\delta_{\epsilon,\mu}\delta_{\mu,\epsilon}=0$, we get
\[
D^2=\begin{pmatrix} -\delta_{\mu,\epsilon}d &0 & 0& 0\\
0& -d\delta_{\mu,\epsilon}-\delta_{\epsilon,\mu}d& 0 & 0\\
0& 0& -d\delta_{\epsilon,\mu}-\delta_{\mu,\epsilon}d & 0\\
0& 0 & 0& -d\delta_{\mu,\epsilon}
\end{pmatrix}.
\]
The operator $D^2$ enjoys the following coercive estimate, 
\begin{equation}
\label{eq_coer}
(D^2\omega,\omega)_{L^2(\Omega M)}\ge C_1\|\omega\|_{H^1(\Omega M)}^2-C_2\|\omega\|_{L^2(\Omega M)}^2,  C_1>0, 
\end{equation}
where $\omega=(\omega^0,\omega^1,\omega^2,\omega^3)$ and $\omega^k\in C^\infty_0(M,\Lambda^k T^* M)$, $k=0,1,2,3$. 

When proving \eqref{eq_coer}, notice that, for $\omega^1\in C^\infty_0(M,\Lambda^1 T^* M)$,
\[
((d\delta_{\mu,\epsilon}+\delta_{\epsilon,\mu}d)\omega^1,\omega^1)_{L^2}=\|\delta_{\mu,\epsilon}\omega^1\|_{L^2}^2+\|d \omega^1\|_{L^2}^2.
\]
An application of Gaffney's inequality 
 gives that 
\[
\|\omega^1\|_{H^1}\le C(M)(\|\omega^1\|_{L^2}+\|d\omega^1\|_{L^2}+\|\delta_{\mu,\epsilon}\omega^1\|_{L^2})
\]
where $C(M)>0$ is a constant. 
The estimate  \eqref{eq_coer} follows, since the treatment of forms of degrees different from $1$ is analogous. See also \cite{Cos1991} for a different proof of coercivity. The ellipticity of $D^2$ now follows from the coercivity estimate  \eqref{eq_coer}, see e.g. \cite{Melin1971}.

\end{proof}

\subsection{Betti numbers and the Euler characteristic of a manifold with boundary}
Let $(M,g_0)$ be an orientable compact Riemannian  manifold of dimension $3$ with boundary. 
The space
\[
\mathcal{H}^k(M)=\{\omega\in L^2(M, \Lambda^k T^* M):d\omega=0, d*_{g_0}\omega=0\}
\]
is called the \emph{space of harmonic fields}. Notice that  this space is infinite dimensional for $1\le k\le 2$, see \cite[Theorem 3.4.2]{Sch95}. 
Moreover, 
it is well-known that harmonic fields are $C^\infty$-smooth in the interior of $M$.
The following two finite dimensional subspaces are distinguished in $\mathcal{H}^k(M)$:
\begin{align*}
\mathcal{H}^k_D(M)=\{\omega\in\mathcal{H}^k(M):\mathbf{t}\omega=0\}\quad\text{and}\\
\mathcal{H}_N^k(M)=\{\omega\in\mathcal{H}^k(M):i^*(\ast_{g_0}\omega)=0\},
\end{align*}
which are called the \emph{Dirichlet} and \emph{Neumann harmonic fileds},  respectively. 
It follows from the Hodge theory that the dimensions of the spaces $\mathcal{H}^k_D(M)$ and $\mathcal{H}^k_N(M)$ are independent of the choice of the metric $g_0$. 
For our purposes, we shall have to specify the choice of the Hodge star operator in the definition of $\mathcal{H}^k(M)$, according to the definition of the codifferential given in \eqref{eq_codiff},
\begin{align*}
\mathcal{H}^2(M)&=\{\omega\in L^2(M, \Lambda^2 T^* M):d\omega=0, d*_{\mu}\omega=0\},\\
\mathcal{H}^k(M)&=\{\omega\in L^2(M, \Lambda^k T^* M):d\omega=0, d*_{\epsilon}\omega=0\},k=1,3.
\end{align*}

Recall \cite{Fran_book} that the space $\mathcal{H}_N^k( M)$ is isomorphic to 
the $k$th homology group of the manifold $H_k(M;\R)$ and   $\mathcal{H}_D^k(M)$ is isomorphic to the $k$th relative homology group $H_k(M,\p M;\R)$. 
The  Poincar\'e-Lefschetz duality states  the existence of the following isomorphism,
\[
H_k(M;\R)\simeq H_{3-k}(M,\p M;\R), \quad k=0,1,2,3.
\]
The $k$th absolute Betti number of the manifold $M$  is given by
\[
\beta_k(M)=\dim \mathcal{H}_N^k(M), \quad k=0,1,2,3,
\]
and the $k$th relative Betti number of $M$ is defined by
\[
\beta_k(M,\p M)=\dim\mathcal{H}_D^k(M), \quad k=0,1,2,3.
\]
Being one of the simplest topological invariants, the Betti numbers carry a basic amount of information about the topology of a manifold in question.
The Betti numbers $\beta_0(M)$ and $\beta_3(M)$ admit a particularly  straightforward geometric interpretation. Namely, 
$\beta_0(M)$ counts the number of the connected components of $M$ and  $\beta_3(M)$ gives the number of the oriented components of $M$ without boundary.  Assuming that the manifold $M$ is connected, we have  $\beta_0(M)=1$ and $\beta_3(M)=0$.
As for the  first Betti number $\beta_1(M)$, it is at least as large as the total number of handles of $\p M$, see
\cite[Theorem 5.1.9]{ColGrigKur}.

The Euler characteristic is defined by
\[
\chi(M)=\beta_{3}(M)-\beta_2(M)+\beta_1(M)-\beta_0(M).
\]
It is known \cite[Corollary 8.8]{Dold_book} that the  Euler characteristics of a compact 3-manifold and its boundary are related by 
 \begin{equation}
 \label{eq_euler_bound}
 \chi(\pM)=2\chi(M).
\end{equation}

Notice finally that if $M$ is a connected compact orientable 3-manifold with vanishing Euler characteristic, then  either the manifold $M$ is closed or
its boundary is a disjoint union of tori.

\subsection{Boundary data for inverse problems}

Let $\Gamma\subset\partial M$ be an open subset of the boundary $\p M$. Consider the following  initial boundary value problem,
\begin{equation}
\label{eq_hyperbolic}
\begin{aligned}
&(\partial_t+D)\omega(x,t)=0\quad\text{in}\quad
M\times\mathbb{R},\\
&\mathbf{t}\omega|_{\partial M\times\mathbb{R}}=f\in
C_0^\infty( \R_-,C^\infty_0(\Gamma, \Lambda T^*M)),\\
&\omega|_{t=-\tau_f}=0,
\end{aligned}
\end{equation}
where  $\tau_f>0$ is such that $\inf \supp (f)>-\tau_f$.
Following \cite{KurLasSom06}, we shall define a solution of  \eqref{eq_hyperbolic} in the following way. Let $E$ be a right inverse to the trace mapping $\mathbf{t}$ such that $Ef(-\tau_f)=0$. We set
\begin{equation}
\label{eq_semi_1}
\omega^f(t)=Ef(t)-\int_{-\tau_f}^t e^{-(t-s)D}(\p_s +D)Ef(s)ds. 
\end{equation}
Here $e^{-t D}$ is the unitary group, generated by the self-adjoint operator $D/i$. 

Associated to the problem \eqref{eq_hyperbolic} is the response operator,
\[
R_{\Gamma}:f\mapsto
\mathbf{n}\omega^f|_{\Gamma\times\R_-}.
\]
The first main result of this work  is the following theorem.

\begin{thm}
\label{thm_main}

Assume that we are  given an open subset $\Gamma\subset \p M$   and  the response operator $R_{\Gamma}$ for any $f\in
C_0^\infty( \R_-,C^\infty_0(\Gamma, \Lambda T^*M))$.  These data determine  the Betti numbers of the manifold $M$.

\end{thm}

Let us now return to  the physical Maxwell's equations
\begin{equation}
\label{eq_Maxwell_phys}
\begin{aligned}
&\omega^1_t=\delta_{\epsilon,\mu}\omega^2,\quad \delta_{\mu,\epsilon}\omega^1=0,\\
&\omega^2_t=-d\omega^1,\quad d\omega^2=0,\\
&\mathbf{t}\omega^1=h\in C_0^\infty( \R_-,C^\infty_0(\Gamma, \Lambda^1 T^*M)),\\
&\omega|_{t<-\tau_h}=0.
\end{aligned}
\end{equation}
As explained in \cite{KurLasSom06}, the solution to \eqref{eq_Maxwell_phys} is obtained from
\eqref{eq_semi_1} by choosing the boundary source $f$ in \eqref{eq_hyperbolic}
as 
\[
f=(0,h,-\int_{-\tau_h}^t dh(t')dt'). 
\]

The response operator for \eqref{eq_Maxwell_phys} is defined by 
\[
\tilde R_\Gamma:h\mapsto
\mathbf{n}\omega^{h,2}|_{\Gamma\times \R_-},
\]
where $\omega^h$ is the solution to \eqref{eq_Maxwell_phys}.
Notice that in the classical terminology of electric and magnetic fields, the response operator $\tilde R_\Gamma$ maps the tangential component of the electric field $n\times E|_{\Gamma\times\R_-}$ to the tangential component of the  magnetic field $n\times H|_{\Gamma\times \R_-}$.

\begin{thm}
\label{thm_main_max}

Given an open subset $\Gamma\subset\p M$ and  the response operator $\tilde R_\Gamma$ for any $h\in
 C_0^\infty( \R_-,C^\infty_0(\Gamma, \Lambda^1 T^*M))$, the first absolute Betti number $\beta_1(M)$ of the manifold $M$ can be determined.

\end{thm}

\begin{cor}
\label{cor_main} The knowledge of the boundary $\p M$ and the response operator $\tilde R_\Gamma$, $\Gamma\subset \p M$,  
 for any $h\in
 C_0^\infty( \R_-,C^\infty_0(\Gamma, \Lambda^1 T^*M))$, determines the first  and  the second  absolute Betti numbers  $\beta_1(M)$ and  $\beta_2(M)$ of $M$.
\end{cor}

Corollary \ref{cor_main} follows from Theorem \ref{thm_main_max} together with  \eqref{eq_euler_bound}.

\subsection{Examples}

The following two examples illustrate 
 the significance of our results for the determination of the topological structure of an unknown object from the boundary measurements. 
This may have applications to practical situations, where the structure of complicated voids in an unknown object is to be recovered.

\begin{exmp}

Let $M\subset \R^3$ be obtained from a large ball  by removing a finite number of pairwise disjoint solid tori. Then the first absolute Betti number of $M$ is equal to the number of the removed solid tori. Thus, measuring the response operator on a portion of the boundary sphere, we can recover the total number of the removed tori. 

\end{exmp}

\begin{exmp}
Consider a solid torus $ST=S^1\times D^2\subset \R^3$,  where $S^1$ is a unit circle and $D^2$ is a closed two-dimensional disc. The boundary of $ST$ is a two dimensional torus and since $D^2$ is contractible, it follows that the first absolute Betti number of  $ST$ is equal to $1$. 
Let $M$ be the connected sum of $k$ copies of solid tori $ST$. Here we may recall that a connected sum of two manifolds, possibly with boundary, is a manifold formed by deleting a ball in the interior of each of the manifolds and gluing together the resulting boundary spheres.
The boundary of  $M$ is a disjoint union of $k$ copies of two-dimensional tori. It is known that for manifolds of dimension three and higher,  the first absolute Betti number of the connected sum is the sum of the first absolute Betti numbers of the summands. Therefore, the first absolute Betti number of $M$ is equal to $k$. 
It follows from our results that performing measurements on a portion of the boundary of the manifold $M$, we are able to recover the total number of the solid tori.

\end{exmp}

\section{Proof of Theorem \ref{thm_main}}

\subsection{Inner products}

 Let 
$\omega^f(t)=\omega^f(x,t)$ be the  solution to \eqref{eq_hyperbolic}. 
We shall need the following Blagovestchenskii type result, see \cite{Blag69} for such results for one-dimensional inverse problems.

\begin{thm}
\label{thm_blagov} For any 
$f,h\in C_0^\infty( \R_-,C^\infty_0(\Gamma, \Lambda T^*M))$, the
knowledge of $\Gamma\subset \p M$ and the response operator $R_\Gamma$ allows us to evaluate the
inner products
\begin{equation}
\label{eq_blagov} (\omega^{f,k}(t),\omega^{h,k}(s))_{L^2}, \quad k=0,1,2,3,\quad
\text{for} \quad  s,t\ge 0.
\end{equation}

\end{thm}

\begin{proof}

From \eqref{eq_semi_1}, we obtain that
\[
\omega^f(t)=\omega^{f_t}(-1), \quad t\ge 0,
\]
where $f_t=f(\cdot+t+1)$, $f_t\in C_0^\infty( \R_-,C^\infty_0(\Gamma, \Lambda T^*M))$. Therefore, the knowledge of the operator $R_\Gamma$ is equivalent to the knowledge of the operator $f\to \mathbf{n}\omega^f|_{\Gamma\times\R}$. 
To prove this theorem we also need the following fact,
\[
\mathbf{t}(d\omega^k)=d(\mathbf{t}\omega^k),\quad  k=0,1,2,3,
\]
see \cite[Proposition 1.2.6]{Sch95}. This implies that
\begin{align*}
\mathbf{n}\delta_{\epsilon,\mu}\omega^2=&\mathbf{t}(\ast_{\epsilon}\ast_{\epsilon}d\ast_{\mu}\omega^2)=d\mathbf{t}(\ast_{\mu}\omega^2)
=d\mathbf{n}\omega^2,\\
\mathbf{n}\delta_{\mu,\epsilon}\omega^3=&\mathbf{t}(\ast_{\mu}(-1)\ast_{\mu}d\ast_\epsilon\omega^3)=-d\mathbf{t}(\ast_\epsilon\omega^3)=-d\mathbf{n}\omega^3.
\end{align*}
Set $I^k(s,t)=(\omega^{f,k}(t),\omega^{h,k}(s))_{L^2}, \quad k=0,\dots,3$. Then using Stokes' formulae, we get
\begin{align*}
(\p_s^2-\p_t^2)I^0(s,t)&=(\omega^{f,0}(t),\p_s^2\omega^{h,0}(s))-(\p_t^2\omega^{f,0}(t),\omega^{h,0}(s))\\
&=-(\omega^{f,0}(t),\delta_{\mu,\epsilon}d\omega^{h,0}(s))+(\delta_{\mu,\epsilon}d\omega^{f,0}(t),\omega^{h,0}(s))\\
&=\langle \mathbf{t}\omega^{f,0}(t),\mathbf{n}d\omega^{h,0}(s)\rangle-\langle \mathbf{t}\omega^{h,0}(s),\mathbf{n}d\omega^{f,0}(t)\rangle\\
&=-\langle \mathbf{t}\omega^{f,0}(t),\partial_s\mathbf{n}\omega^{h,1}(s)\rangle+\langle \mathbf{t}\omega^{f,0}(t),\mathbf{n}\delta_{\epsilon,\mu}\omega^{h,2}(s)\rangle\\
&+
\langle \mathbf{t}\omega^{h,0}(s),\partial_t\mathbf{n}\omega^{f,1}(t)\rangle-\langle \mathbf{t}\omega^{h,0}(s),\mathbf{n}\delta_{\epsilon,\mu}\omega^{f,2}(t)\rangle\\
&=-\langle \mathbf{t}\omega^{f,0}(t),\partial_s\mathbf{n}\omega^{h,1}(s)\rangle+\langle \mathbf{t}\omega^{f,0}(t),d\mathbf{n}\omega^{h,2}(s)\rangle\\
&+
\langle \mathbf{t}\omega^{h,0}(s),\partial_t\mathbf{n}\omega^{f,1}(t)\rangle-\langle \mathbf{t}\omega^{h,0}(s),d\mathbf{n}\omega^{f,2}(t)\rangle.
\end{align*}
Similarly, 
\begin{align*}
(\p_s^2-\p_t^2)I^1(s,t)&=(\omega^{f,1}(t),\p_s^2\omega^{h,1}(s))-(\p_t^2\omega^{f,1}(t),\omega^{h,1}(s))\\
&=-\langle\p_s\mathbf{t}\omega^{h,0}(s),\mathbf{n}\omega^{f,1}(t)\rangle-
\langle\mathbf{t}\omega^{f,1}(t),\p_s\mathbf{n}\omega^{h,2}(s)+d\mathbf{n}\omega^{h,3}(s)\rangle\\
&+\langle\p_t\mathbf{t}\omega^{f,0}(t),\mathbf{n}\omega^{h,1}(s)\rangle+
\langle\mathbf{t}\omega^{h,1}(s),\p_t\mathbf{n}\omega^{f,2}(t)+d\mathbf{n}\omega^{f,3}(t)\rangle,
\end{align*}
\begin{align*}
(\p_s^2-\p_t^2)I^2(s,t)&=(\omega^{f,2}(t),\p_s^2\omega^{h,2}(s))-(\p_t^2\omega^{f,2}(t),\omega^{h,2}(s))\\
&=-\langle\p_s\mathbf{t}\omega^{h,1}(s)+d\mathbf{t}\omega^{h,0}(s),\mathbf{n}\omega^{f,2}(t)\rangle-
\langle\mathbf{t}\omega^{f,2}(t),\p_s\mathbf{n}\omega^{h,3}(s) \rangle\\
&+\langle \p_t\mathbf{t}\omega^{f,1}(t)+d\mathbf{t}\omega^{f,0}(t),\mathbf{n}\omega^{h,2}(s)\rangle
+\langle \mathbf{t}\omega^{h,2}(s),\p_t\mathbf{n}\omega^{f,3}(t) \rangle,
\end{align*}
\begin{align*}
(\p_s^2-\p_t^2)I^3(s,t)&=(\omega^{f,3}(t),\p_s^2\omega^{h,3}(s))-(\p_t^2\omega^{f,3}(t),\omega^{h,3}(s))\\
&=-\langle\p_s\mathbf{t}\omega^{h,2}(s)+d\mathbf{t}\omega^{h,1}(s),\mathbf{n}\omega^{f,3}(t)\rangle\\
&+
\langle\p_t\mathbf{t}\omega^{f,2}(t)+d\mathbf{t}\omega^{f,1}(t),
\mathbf{n}\omega^{h,3}(s)\rangle.
\end{align*}
Hence $I^k(s,t)$, $k=0,1,2,3$, satisfies  an inhomogeneous one-dimensional wave equation in the unbounded region $\{(s,t)\in \R^2:s\ge -\tau_h,t\ge -\tau_f\}$,  whose
right hand side is determined from the knowledge of   $\Gamma$ and $R_\Gamma$.
Since 
\[
I^
k(-\tau_h,t)= I^k(s,-\tau_f)=0,\quad \p_s I^k(-\tau_h,t)=\p_tI^k(s,-\tau_f)=0,
\] 
we can determine 
$I^k(s,t)$ in the entire region $s\ge -\tau_f$, $t\ge -\tau_f$. The result follows. 
\end{proof}

\subsection{Controllability result}
In the isotropic setting and the case when $\epsilon(x)=\alpha(x)\mu(x)$, $\alpha(x)>0$, 
one can use a generalization of Tataru's unique continuation theorem \cite{EllNakTat02, KurLasSom06} 
to  obtain controllability results with sources supported in a finite time interval. 
As already mentioned in the introduction, such unique continuation results do not seem to be available in the general anisotropic setting. Nevertheless, we shall next show that partial controllability results in the general anisotropic setting on an infinite time interval 
can be obtained using a unique continuation principle for elliptic systems.   As shown below, this turns out to be sufficient for the reconstruction of the Betti numbers.

Let $\mathcal{H}_D(M):=\oplus_{k=0}^3\mathcal{H}_D^k(M)$ be the space of all  Dirichlet harmonic fields, 
and let
$\Pi:L^2(M, \Lambda T^* M)\to \mathcal{H}_D(M)$ be the orthogonal projection.

\begin{thm}
\label{thm_controllability_new} We have
\[
\{\Pi (\omega^{f}(0)):f\in C_0^\infty(\R_-,C^\infty_0(\Gamma,\Lambda T^* M))\}= \mathcal{H}_D( M).
\]
\end{thm}


\begin{proof} Let $\eta\in  \mathcal{H}_D( M)$. If we prove that the orthogonality condition 
\[
(\omega^{f}(0),\eta)_{L^2}=0 \quad \text{for all } f\in C_0^\infty(\R_-,C^\infty_0(\Gamma,\Lambda T^* M))
\]
implies that  $\eta=0$, then the space  $\{\Pi\omega^{f}(0)):f\in C_0^\infty(\R_-,C^\infty_0(\Gamma,\Lambda T^* M))\}$ is dense in  $\mathcal{H}_D(M)$. Since the latter space is finite dimensional, the claim follows.

As $D\eta =0$, we shall view $\eta(x)$ as the  solution  to the following problem, dual to \eqref{eq_hyperbolic},
\begin{equation}
\label{eq_dual_new}
\begin{aligned}
&(-\partial_t-D)u=0, \quad\text{in}\quad
M\times\mathbb{R},\\
& \mathbf{t}u|_{\partial M\times\mathbb{R}}=0,\\
& u|_{t=0}=\eta.
\end{aligned}
\end{equation}
Using \eqref{eq_stokes_D}, we have
\begin{align*}
\p_t(\omega^f,u)_{L^2}=-(D\omega^f,u)_{L^2}-(\omega^f, Du)_{L^2}=-\langle f,\mathbf{n}u\rangle.
\end{align*}
Thus, 
\[
\int_{-\tau_f}^0 \langle f,\mathbf{n}u\rangle dt=-(\omega^f(0),\eta)_{L^2}+(\omega^f(-\tau_f),u(-\tau_f))_{L^2}=0.
\] 
The choice of $-\tau_f$ implies that 
\[
\int_{\mathbb{R}_-}\langle f,\mathbf{n} u\rangle dt=0
\]
for all $f\in C_0^\infty(\R_-,C^\infty_0(\Gamma,\Lambda T^* M))$. Thus, $\mathbf{n}u=0$ on $\Gamma\times\mathbb{R}_-$. 

Now if $\Gamma$ coincides with the whole boundary of the manifold $M$, then we are done, since $\mathcal{H}_D^k(M)\cap\mathcal{H}_{N}^k(M)=\{0\}$,  see \cite[p. 130]{Sch95}. 

In the case when $\Gamma$ is a proper  open subset of $\p M$, we proceed as follows. Notice that $\eta(x)$ solves the second order elliptic system $D^2\eta=0$ on $M$ with zero Cauchy data on $\Gamma$,
\[
(\mathbf{t}\eta,\mathbf{n}\eta,\mathbf{t}\delta \eta,\mathbf{n}d\eta),
\]
where $\delta\eta^k=\delta_{\mu,\epsilon}\eta^k$, $k=1,3$, and $\delta\eta^k=\delta_{\epsilon, \mu}\eta^k$, $k=2$. Thus, by the unique continuation principle for second order elliptic systems with  diagonal principal  part, see \cite[Theorem 4.3]{Isakov04}, we get $\eta=0$ in $M$. 

\end{proof}

\begin{cor}
\label{cor_1}
Let $\Pi^k:L^2(M, \Lambda^k T^* M)\to \mathcal{H}_D^k(M) $ be  the orthogonal projection onto the space of the Dirichlet harmonic $k$-fields. 
Then
\[
\{\Pi^k(\omega^{f,k}(0)): f\in C_0^\infty(\R_-,C^\infty_0(\Gamma,\Lambda T^* M))\}= \mathcal{H}_D^k(M),\quad k=0,1,2,3.
\]

\end{cor}

\subsection{Determination of the  Betti numbers of the manifold}

\label{subsection_betti}

\begin{lem} 
\label{lem_inner_2}
Let $f,h\in C_0^\infty(\R_-,C^\infty_0(\Gamma,\Lambda T^* M))$. Then given the response operator $R_\Gamma$, it is possible to find the inner products
\[
(\Pi ^k \omega^{f,k}(0), \omega^{h,k}(0))_{L^2}, \quad k=0,1,2,3.
\]
\end{lem}

\begin{proof}

Using \eqref{eq_stokes_formulae}, we check by a direct computation that $\ker D=\mathcal{H}_D(M)$. We can therefore view $\Pi$ as the spectral projection onto the zero eigenspace of $D$.  Consider the unitary group $e^{-tD}$, $t\in \R$, on $L^2$. We shall make use of the following essentially well-known formula,
\begin{equation}
\label{eq_pi}
\Pi=\lim_{T\to +\infty}\frac{1}{T}\int_0^T e^{-tD}dt,
\end{equation}
valid in the sense of strong convergence of operators. When checking \eqref{eq_pi}, let
\[
\Pi_T=\frac{1}{T}\int_0^T e^{-tD}dt\in \mathcal{L}(L^2,L^2).
\]
Since $\|\Pi_T\|_{\mathcal{L}(L^2,L^2)}\le 1$, it suffices to check that $\Pi_Tx\to \Pi x$ when $x$ varies in a dense subset of $L^2$. We can take this subset to be the set of all finite linear combinations of the eigenfunctions of $D$. To get  \eqref{eq_pi}, we only need to observe that when $\lambda\in \R$,
\[
\lim_{T\to+\infty}\frac{1}{T}\int^T_0e^{it\lambda}dt=\begin{cases} 1 & \text{if }\lambda=0,\\
0 & \text{if }\lambda\ne 0.
\end{cases}
\]
Now notice that since $\supp(f)\subset \R_-$, we have 
\[
e^{-tD}\omega^f(0)=\omega^f(t),\quad t\ge 0,
\] 
and therefore,
\[
\Pi\omega^f(0)=\lim_{T\to +\infty}\frac{1}{T}\int_0^T\omega^f(t)dt. 
\]
We get
\begin{align*}
(\Pi ^k \omega^{f,k}(0), \omega^{h,k}(0))_{L^2}&=\lim_{T\to+\infty}\frac{1}{T}\int_0^T
(\omega^{f,k}(t),\omega^{h,k}(0))_{L^2}\\
&=\lim_{T\to+\infty}\frac{1}{T}\int_0^T
(\omega^{f_t,k}(0),\omega^{h,k}(0))_{L^2},
\end{align*}
where $f_t=f(\cdot+t)$, $f_t\in C_0^\infty(\R_-,C^\infty_0(\Gamma,\Lambda T^* M))$. Here we have used that $\omega^f(t)=\omega^{f_t}(0)$, as follows from
\eqref{eq_semi_1}. 
An application of  Theorem \ref{thm_blagov}  concludes the proof.

\end{proof}

We have the following result with implies Theorem \ref{thm_main}. 

\begin{lem}
Given $\Gamma\subset \p M$ and the response operator $R_\Gamma$,  it is possible to construct a finite number of boundary sources $f_j\in C_0^\infty(\R_-,C^\infty_0(\Gamma,\Lambda T^* M))$ such that $\Pi^k(\omega^{f_j,k}(0))$ form a basis of $\mathcal{H}_{D}^k(M)$, $0\le k\le 3$. 
\end{lem}

\begin{proof}
Let $\{h_j\}_{j=1}^\infty$ be a dense countable set in $C_0^\infty(\R_-,C^\infty_0(\Gamma,\Lambda T^* M))$. We can use the Gram-Schmidt orthogonalization procedure to construct the sources $f_j$. More precisely, we define $f_j$ recursively by
\begin{align*}
f_1& =\frac{h_1}{(\Pi^k \omega^{h_1,k}(0),\omega^{h_1,k}(0))_{L^2}^{1/2}},\\
g_j& =h_j-\sum_{i=1}^{j-1}(\Pi^k \omega^{h_j,k}(0),\omega^{h_i,k}(0))_{L^2}f_i,\quad j=2,3,\dots,\\
f_j & =\frac{g_j}{(\Pi^k \omega^{g_j,k}(0),\omega^{g_j,k}(0))_{L^2}^{1/2}}.
\end{align*}
When $g_j = 0$, we remove the corresponding $h_j$ from the
original sequence and continue the procedure. The number of sources $f_j$ produced by the Gram-Schmidt orthogonalization procedure will then be the dimension of $\mathcal{H}_{D}^k(M)$, according to Corollary \ref{cor_1}.
 
\end{proof}

\section{Proof of Theorem \ref{thm_main_max}}

First notice that as in Theorem \ref{thm_blagov}, for any $f,h\in C_0^\infty(\R_-,C^\infty_0(\Gamma,\Lambda^1 T^* M))$, the knowledge of the response operator $\tilde R_\Gamma$ allows us to evaluate the inner products,
\[
(\omega^{f,k}(t),\omega^{h,k}(s))_{L^2}, \quad k=1,2,\quad t,s\ge 0,
\]
where $\omega^f, \omega^h$ are solutions of physical Maxwell's equations \eqref{eq_Maxwell_phys}. 

We have the following controllability result. 

\begin{lem} 
Let $\omega^f$ be a solution to physical Maxwell's equations \eqref{eq_Maxwell_phys}.  Then 
\[
\{\Pi^2(\omega^{f,2}(0)):f\in C_0^\infty(\R_-,C^\infty_0(\Gamma,\Lambda^1 T^* M))\}=\mathcal{H}_{D}^2(M),
\]
where $\Pi^2$ is the orthogonal projection onto the space of the Dirichlet harmonic $2$-fields. 
\end{lem}

\begin{proof}
Let $\eta^2\in \mathcal{H}_{D}^2(M)$. Assume that 
\[
(\omega^{f,2}(0),\eta^2)_{L^2}=0\quad\text{for all} \ f\in C_0^\infty(\R_-,C^\infty_0(\Gamma,\Lambda^1 T^* M)). 
\]
Now \eqref{eq_Maxwell_phys} and Stokes' formula imply that
\begin{align*}
\p_t(\omega^{f,2}(t),\eta^2)_{L^2}&=(-d\omega^{f,1}(t),\eta^2)_{L^2}=-(\omega^1(t),\delta_{\epsilon,\mu}\eta^2)_{L^2}-\langle \mathbf{t}\omega^{f,1}(t),\mathbf{n}\eta^2 \rangle\\
&=-\langle f(t),\mathbf{n}\eta^2 \rangle.
\end{align*}
Thus,
\[
\int_{\R_-}\langle f(t),\mathbf{n}\eta^2 \rangle=-(\omega^{f,2}(0),\eta^2)_{L^2}+(\omega^{f,2}(-\tau_f),\eta^2)_{L^2}=0,
\]
for all $f\in C_0^\infty(\R_-,C^\infty_0(\Gamma,\Lambda^1 T^* M))$.
Hence, $\mathbf{n}\eta^2=0$ on $\Gamma$. Moreover, $\Delta\eta^2=0$ on $M$ and $\mathbf{t}\eta^2=0$ on $\p M$. By the unique continuation principle,  we get  $\eta^2=0$.

\end{proof}

Proceeding further as in Subsection \ref{subsection_betti}, we can  recover the first absolute Betti number $\beta_1(M)$ from the knowledge of $\Gamma$ and $\tilde R_\Gamma$. This completes the proof of Theorem \ref{thm_main_max}.

\section{Acknowledgements}
We would to thank Semen Podkorytov for a helpful discussion and providing useful references on topology of manifolds. 
The research of K.K. was financially supported by the
Academy of Finland (project 125599). The research of Y.K. is partially
supported by EPSRC Grant EP/F034016/1. The research of M.L. 
was financially supported 
 by the Academy of Finland Center of Excellence programme 213476.


\begin{thebibliography} {1}



\bibitem{BelIsaPesShar200}
Belishev, M.,  Isakov, V.,   Pestov, L., and  Sharafutdinov, V.,  \emph{On the reconstruction
of a metric from external electromagnetic measurements} (Russian)  Dokl. Akad. Nauk  \textbf{372} 
(2000),  no. 3, 298--300.

\bibitem{BelSha2008}
Belishev, M.,  Sharafutdinov, V., \emph{Dirichlet to Neumann operator on differential forms},
Bull. Sci. Math. \textbf{132} (2008), 128--145.

\bibitem{Blag69}
Blagovestchenskii, A. S., \emph{A one-dimensional inverse boundary value problem
  for a second order hyperbolic equation}, Zap. Nau\v cn. Sem. Leningrad.
  Otdel. Mat. Inst. Steklov. (LOMI) \textbf{15} (1969), 85--90.



\bibitem{CarOlaSalo} 
Caro P.,  Ola P.,  and Salo M., \emph{Inverse boundary value problem for Maxwell equations with local data},
Comm. PDE \textbf{34} (2009), no. 11, 1425--1464.


\bibitem{ColGrigKur}
Collins, D. J., Grigorchuk, R. I., Kurchanov, P. F. and Zieschang, H., \emph{Combinatorial group theory and applications to geometry}.
Translated from the 1990 Russian original by P. M. Cohn. Reprint of the original English edition from the series Encyclopaedia of Mathematical Sciences [Algebra. VII, Encyclopaedia Math. Sci., 58, Springer, Berlin, 1993]. Springer-Verlag, Berlin, 1998, 240 pp.


\bibitem{ColPai92}
Colton, D., P\"aiv\"arinta, L.,  \emph{The uniqueness of a solution to an inverse scattering
problem for electromagnetic waves},  Arch. Rational Mech. Anal.  \textbf{119}  (1992),  no. 1,
59--70.

\bibitem{Cos1991}
Costabel, M.,  \emph{A coercive bilinear form for Maxwell's equations},  J. Math. Anal. Appl.  \textbf{157}  (1991),  no. 2, 527--541.

\bibitem{Dold_book}
Dold, A.,  \emph{Lectures on algebraic topology}. Second edition. Grundlehren der Mathematischen Wissenschaften [Fundamental Principles of Mathematical Sciences], 200. Springer-Verlag, Berlin-New York, 1980, 377 pp. 

\bibitem{EllNakTat02}
Eller, M.,  Isakov, V.,  Nakamura, G.,  and Tataru, D.,  \emph{Uniqueness and stability in the Cauchy problem for Maxwell and elasticity systems}.  Nonlinear partial differential equations and their applications,  329--349, Stud. Math. Appl., 31, North-Holland, Amsterdam, 2002. 

\bibitem{Fran_book}
Frankel, T.,  \emph{The geometry of physics}. An introduction. Cambridge University Press, Cambridge, 1997. 654 pp.


\bibitem{GreKurLasUhl2009}
Greenleaf, A., Kurylev, Y., Lassas, M., and  Uhlmann, G., \emph{Invisibility and
inverse problems},  Bull. Amer. Math. Soc. (N.S.)  \textbf{46}  (2009),  no. 1, 55--97.



\bibitem{GreKurLasUhl2007}
Greenleaf, A., Kurylev, Y., Lassas, M., and  Uhlmann, G., \emph{Full-wave invisibility of active devices at all frequencies},  Comm. Math. Phys.
\textbf{275} (2007), no. 3, 749--789. 

  
\bibitem{Isakov04} Isakov, V.,
     \emph{Carleman type estimates and their applications},
 New analytic and geometric methods in inverse problems,
    Springer, Berlin, 2004, 93--125.
    


\bibitem{KenSalUhl}
Kenig C., Salo M., and Uhlmann G., \emph{Inverse problems for the anisotropic Maxwell equations}, preprint.


\bibitem{KurLass06}
 Kurylev, Y., Lassas, M., \emph{Inverse problems and index formulae for {D}irac
  operators}, Advances in Mathematics \textbf{221} (2009), 170--216.


\bibitem{KurLasSom06}
Kurylev, Y., Lassas, M., and Somersalo, E.,  \emph{Maxwell's equations with a polarization independent wave velocity: direct and inverse problems},  J. Math. Pures Appl. (9)  \textbf{86}  (2006),  no. 3, 237--270  

\bibitem{LasUhl01}
Lassas, M.,  Uhlmann, G.,  \emph{On determining a Riemannian manifold from the
Dirichlet-to-Neumann map}.  Ann. Sci. \'Ecole Norm. Sup. (4)  \textbf{34}  (2001),  no. 5, 771--787.

\bibitem{LasTayUhl03}
Lassas, M., Taylor, M., and Uhlmann, G., \emph{The Dirichlet-to-Neumann map for
complete Riemannian manifolds with boundary},  Comm. Anal. Geom.  \textbf{11}  (2003),  no. 2,
207--221.

\bibitem{McD97}
McDowall, S.,  \emph{Boundary determination of material parameters from electromagnetic
boundary information},  Inverse Problems  \textbf{13}  (1997),  no. 1, 153--163

\bibitem{Melin1971}
Melin, A., \emph{Lower bounds for pseudo-differential operators},
Ark. Mat. \textbf{9} (1971), 117--140. 


\bibitem{OlaPaiSom1993}
Ola P., P\"aiv\"arinta L., and Somersalo E., \emph{An inverse boundary value problem in electrodynamics}, Duke Math. J. 70 (1993), 617--653.



\bibitem{OlaPaiSom2003}
Ola P., P\"aiv\"arinta L., and Somersalo E., \emph{Inverse problems for time harmonic electrodynamics}, Inside out: inverse problems, MSRI publications \textbf{47} (2003), 169--191.

\bibitem{OlaSom1996}
Ola P., Somersalo E., \emph{Electromagnetic inverse problems and generalized Sommerfeld potentials}, SIAM J. Appl. Math. \textbf{56} (1996), 1129--1145.


\bibitem{Paquet} 
Paquet, L., \emph{Probl\`emes mixtes pour le syst\`eme de Maxwell}. (French) [Mixed problems for the Maxwell system] Ann. Fac. Sci. Toulouse Math. (5) \textbf{4} (1982), no. 2, 103--141.


\bibitem{Sch95} Schwarz, G., \emph{Hodge decomposition---a method for solving boundary
  value problems}, Lecture Notes in Mathematics, vol. 1607, Springer-Verlag,
  Berlin, 1995.

\bibitem{Thi79}
Thirring, W., \emph{A course in mathematical physics}. Vol. 2.
Classical field theory. Translated from the German by Evans M. Harrell. Springer-Verlag, New York-Vienna, 1979, 249 pp.

\end{thebibliography}

\end{document}